\documentclass[12pt]{amsart}

\usepackage{amssymb,amsfonts}
\usepackage{amsthm}
\usepackage{graphicx}
\usepackage[all,arc]{xy}
\usepackage[colorlinks=true, allcolors=blue]{hyperref}
\usepackage{enumerate}
\usepackage{bbm}
\usepackage{dsfont}
\usepackage{mathrsfs}
\usepackage{tikz-cd}
\usepackage{import}
\usepackage[utf8]{inputenc}
\usepackage{hyperref}
\usepackage{tikz}

\usepackage[left=1in,top=1in,right=1in,bottom=1in]{geometry}

\usepackage{mathtools}

\newtheorem{thm}{Theorem}[section]

\newtheorem{prop}[thm]{Proposition}
\newtheorem{lem}[thm]{Lemma}

\theoremstyle{definition}
\newtheorem{defn}[thm]{Definition}

\newtheorem{example}[thm]{Example}

\newtheorem{rmk}[thm]{Remark}
\newtheorem{fct}[thm]{Fact}

\newcommand{\gl}{\operatorname{GL}}
\newcommand{\im}{\operatorname{Im}}
\newcommand{\ob}{\operatorname{Ob}}
\newcommand{\aut}{\operatorname{Aut}}
\newcommand{\Aut}{\operatorname{Aut}}
\newcommand{\id}{\mathrm{Id}}
\newcommand{\colim}{\varinjlim}
\newcommand{\ifl}{\operatorname{ifl}}
\newcommand{\dfl}{\operatorname{dfl}}

\def\Q{\mathbb{Q}}
\def\R{\mathbb{R}}

\def\Z{\mathbb{Z}}

\title{Discrete Morse theory for symmetric Delta-complexes}
\author{Claudia He Yun}
\address{MPI MiS, 04103 Leipzig, Germany}
\email{\url{clyun@mis.mpg.de}}

\begin{document}

\begin{abstract}
We generalize Forman's discrete Morse theory to the context of symmetric $\Delta$-complexes. As an application, we prove that the coloop subcomplex of the link of the origin $LA^{\mathrm{trop},\mathrm{P}}_g$ in the moduli space of principally polarized tropical abelian varieties of dimension $g$ with respect to the perfect cone decomposition is contractible.
\end{abstract}
	
\maketitle

\section{Introduction}
In this paper, we develop a discrete Morse theory for symmetric $\Delta$-complexes. First introduced in \cite{Forman+discrete-morse}, discrete Morse theory is an analogue of Morse theory for CW-complexes. It considers a discrete Morse function on a CW-complex and produces another CW-complex with the same homotopy type but with fewer cells. It has seen many applications to topological spaces that are combinatorially defined, see, for example, \cite{Chari, Babson-Kozlov, Babson-bjorner-etal,Babson+Hersh+discrete-morse-and-shellability} and is an essential tool in computational topology \cite{Benedetti-Lutz-randomDMF,Bunnett-Joswig-Pfeifle}.

On the other hand, symmetric $\Delta$-complexes and symmetric CW-complexes are generalizations of the usual cell complexes. They are constructed from quotients of standard cells. They have played important roles in recent developments of tropical geometry. A notable example of symmetric $\Delta$-complexes is the link of the moduli space of tropical curves $\Delta_{g,n}$ of genus $g$ with $n$ marked points, which is the central object studied by a series of papers (\cite{CGP1,CGP2,yun2021sn,BCGY}).

A common theme in studying symmetric cell complexes is to identify contractible or acyclic subcomplexes \cite{CGP1,ACP-symCW,brandt2020top,kannan2020topology}. We add to this effort by generalizing contractibility criteria from discrete Morse Theory to the context of symmetric $\Delta$-complexes. We define a discrete Morse function analogous to the usual discrete Morse function and prove the following theorem parallel to the main theorem of Forman's discrete Morse theory. Let $B^p$ be the closed unit ball in $\R^p$ and $O(p)$ be the orthogonal group in dimension $p$.
 
\begin{thm}
Let $X$ be a finite, nonempty symmetric $\Delta$-complex and $f$ a discrete Morse function on it. Then the geometric realization $|X|$ is homotopy equivalent to a symmetric CW-complex with exactly one $p$-cell $(B^p)^\circ/G$ for each critical symmetric orbit $[\alpha]$ of dimension $p$ with respect to $f$, where $G\subset O(p)$ is finite and isomorphic to $\aut(\sigma)$.
\label{thm: main theorem}
\end{thm}

\cite{Chari} gives a combinatorial interpretation of discrete Morse theory in terms of matchings on the Hasse diagram of the face poset of a given cell complex. We prove a similar theorem for symmetric $\Delta$-complexes. We define the notion of a permissible matching in Section \ref{sec: dmt combo description}.

\begin{thm}
Let $X$ be a finite, nonempty symmetric $\Delta$-complex and $H_X$ the Hasse diagram of the poset of symmetric orbits of $X$, not including the empty set, regarded as a graph. Let $M$ be a permissible acyclic matching on $H_X$. Then $|X|$ is homotopy equivalent to a symmetric CW-complex with exactly one $p$-cell $(B^p)^\circ/G$ for each symmetric orbit of $p$-simplices $[\sigma]$ that is not matched in $M$, where $G\subset O(p)$ is finite and isomorphic to $\aut(\sigma)$.
\label{thm: main thm combo}
\end{thm}

As an application, we study the moduli spaces $A^{\mathrm{trop},\mathrm{P}}_g$ of principally polarized tropical abelian varieties of dimension $g$ with respect to the perfect cone decomposition. Denote the link of the origin of $A^{\mathrm{trop},\mathrm{P}}_g$ by $LA^{\mathrm{trop},\mathrm{P}}_g$. We strengthen \cite[Theorem 5.23]{brandt2020top}, which states that the coloop subcomplex of $LA^{\mathrm{trop},\mathrm{P}}_g$ is acyclic, i.e., it has zero reduced rational homology, to the following.
\begin{thm}
The coloop subcomplex of $LA^{\mathrm{trop},\mathrm{P}}_g$ is contractible.
\label{thm: coloop contractible}
\end{thm}

\subsection{Related work}
\begin{enumerate}
    \item Freij develops an equivariant discrete Morse theory for simplicial complexes with group actions \cite{Freij-eqvDMT}. This is different from our setup because in general there is no action of a single group on the entire symmetric $\Delta$-complex, even though each symmetric orbit admits an action of a symmetric group.
    \item Symmetric CW-complexes can be obtained as links of the origin of generalized cone complexes \cite{ACP-symCW}. In \cite{Bunnett-Joswig-Pfeifle}, the authors study symmetric CW-complexes through discrete Morse theory by considering the second barycentric subdivision of the symmetric CW-complexes. In general, the second barycentric subdivision of a symmetric $\Delta$-complex is a simplicial complex, for which there are many standard tools in combinatorial topology. However, one drawback of taking barycentric subdivisions is that the number of cells to be considered grows significantly.
    \item \cite{CGP2} develops a contractibility criterion for symmetric $\Delta$-complexes that focuses on vertices and prescribes deformation retractions based on containment of specific vertices. This contractibility theory has a similar flavor to our generalized discrete Morse theory. The two approcahes allow for different types of deformation retractions. Ours, like Forman's discrete Morse theory, works by sequentially collapsing simplices along their codimension-one faces. The contractibility criterion in \cite{CGP2} uses the inclusion and exclusion of certain vertices and is able to produce more general collapses. It would be interesting to see how these two theories are related.  
\end{enumerate}

\subsection{Outline of the paper} In Section \ref{sec: sym delta cx}, we review the necessary background on symmetric $\Delta$-complexes and give the definition of a permissible face. In Section \ref{sec: dmt}, we develop a generalized discrete Morse Theory. We discuss two approaches: using discrete Morse functions and using a combinatorial criterion on the poset of symmetric orbits of a symmetric $\Delta$-complex. Lastly, in Section $\ref{sec: Ag}$, we apply our theory to $A^{\mathrm{trop},\mathrm{P}}_g$ and show the coloop subcomplex of the link of the origin is contractible.

\subsection*{Acknowledgements} We thank Melody Chan for suggesting this direction of research and many useful discussions. We thank Sam Payne for suggesting the coloop subcomplex as an application and Madeline Brandt for answering many questions on $A^{\mathrm{trop},\mathrm{P}}_g$. We thank Michael Joswig for useful conversations.

\section{Symmetric $\Delta$-complexes}
\label{sec: sym delta cx}

In this section, we review the construction and properties of symmetric $\Delta$-complexes. We define the notion of a permissible face (Definition \ref{defn: permissible face}), which is a key definition that did not appear in Forman's discrete Morse theory. Denote the set $\{0,1,\cdots,p\}$ by $[p]$.

\begin{defn}
Let $I$ be the category whose objects are $[p]$ for each $p \geq 0$ together with $[-1] := \emptyset$ and whose morphisms are all injections. A symmetric $\Delta$-complex is a functor $X: I^{\text{op}} \to \text{Sets}$.
\label{def: sym Delta cx}
\end{defn}

\begin{rmk}
Let $i_p: [p] \to [p+1]$ be inclusion. Morphisms in $I$ are generated by permutations $[p] \to [p]$ and $i_p$.
\end{rmk}

Recall that the standard $p$-simplex $\Delta^p \subset \R^{p+1}$ is defined to be the convex hull of the standard basis vectors $e_0,\dots,e_p$, i.e., \[\Delta^p = \left\{\sum_{i=0}^p t_ie_i : \sum t_i = 1 \text{ and } t_i \geq 0 \text{ for all }i\right\}.\] Given a set map $\theta: [p] \to [q]$, it induces a map between simplices $\theta_*:\Delta^p \to \Delta^q$ where \[\theta_*\left(\sum_{i=0}^p t_ie_i\right) = \sum_{i=0}^q \left(\sum_{j \in \theta^{-1}(i)}t_j\right)e_i.\]

Let $X$ be a symmetric $\Delta$-complex. For notational simplicity, we denote $X([p])$ by $X_p$. The geometric realization of $X$ is the topological space 
\begin{equation}
|X| = \left(\coprod_p X_p \times \Delta^p\right)/\sim,
\label{eqn: geometric realization}
\end{equation}
where the equivalence relation is given by $(x,\theta_*a) \sim (\theta^*x,a)$, where $x\in X_p$, $\theta: [p] \to [q]$ an injection, and $a \in \Delta^p$. This equivalence relation is analogous to the gluing maps of a $\Delta$-complex.

We refer to an element $\alpha \in X_p$ as a $p$-simplex and write $[\alpha]$ to represent its symmetric orbit, i.e., the set of simplices $\{\rho^*(\alpha):\rho\in S_{p+1}\}$. For a $p$-simplex $\alpha$, we sometimes write $\alpha^{(p)}$ to emphasize its dimension. If an injection $\theta:[p] \to [q]$ induces a map $\theta^*(\beta) = \alpha$, we say that $\alpha$ is a face of $\beta$ and $\theta$ glues $\alpha$ to $\beta$; we write $\alpha < \beta$ in this case. Face relation is well-defined for symmetric orbits, as stated in the following lemma.

\begin{lem}
Suppose we have simplices $\alpha^{(p)} < \beta^{(q)}$, $\alpha' \in [\alpha]$ and $\beta' \in [\beta]$. Then $\alpha' < \beta'$.
\label{lem: face relation holds for sym orbits}
\end{lem}
Note that the face relation is not a partial order on the set of simplices. It is not antisymmetric: two distinct simplices $\alpha$ and $\alpha'$ in the same symmetric orbit satisfy both $\alpha \leq \alpha'$ and $\alpha' \leq \alpha$. However, the face relation induces a partial order on the set of symmetric orbits of simplices.

\begin{defn}
We define the automorphism group of a simplex $\alpha \in X_p$ to be $\Aut(\alpha) = \{\rho \in S_{p+1}: \rho^*(\alpha) = \alpha\}$.
\label{defn: aut of a simplex}
\end{defn}

\noindent We now give the definition of permissibility.

\begin{defn}
Let $X$ be a symmetric $\Delta$-complex. Suppose $\rho$ glues $\alpha^{(p)}$ to $\beta^{(p+1)}$. We say $\alpha$ is a \textit{permissible} face of $\beta$ if they satisfy the following conditions:
\begin{enumerate}
    \item Any two injections $\theta:[p]\to[p+1]$ and $\theta':[p]\to[p+1]$ such that $\theta^*(\beta)=(\theta')^*(\beta)=\alpha$ satisfy $\im \theta = \im \theta'$.
    \item The function $f_\rho:\Aut(\beta) \to \Aut(\alpha)$ given by $f(\pi) = \rho^{-1}|_{\im \rho}\circ \pi \circ \rho$ induces an isomorphism.
\end{enumerate}
\label{defn: permissible face}
\end{defn}

Implicit in the definition of $f_\rho(\pi)$ above is that $\pi$ restricts to a permutation on $\im \rho$. We verify this claim under assumption (1).

\begin{proof}
Let $\pi \in \Aut(\beta)$ and let $\rho:[p] \to [p+1]$ be an injection such that $\rho^*(\beta) = \alpha$. Suppose $\pi$ does not restrict to a permutation on $\im \rho$. Then $\pi\circ\rho$ is another map that glues $\alpha$ to $\beta$ but $\im(\pi\circ\rho) \neq \im \rho$. This contradicts item (1). Therefore, $f_\rho(\pi)$ exists and satisfies that $\pi\circ\rho = \rho\circ f_\rho(\pi)$. So $f_\rho(\pi) \in \Aut(\alpha)$.
\end{proof}

\begin{lem}
The function $f_\rho$ is unique up to conjugation by a permutation of $[p]$.
\end{lem}

\begin{proof}
Suppose $\rho,\rho':[p]\to [p+1]$ are both maps that glues $\alpha$ to $\beta$. Then there is a permutation $i:[p] \to [p]$ such that $\rho' = \rho \circ i$. So we have $f_{\rho'}(\pi) = i^{-1}\circ \rho^{-1}|_{\im \rho} \circ \pi \circ \rho \circ i$.
\end{proof}

\noindent We show that permissibility is well-defined for symmetric orbits.

\begin{lem}
Suppose $\alpha^{(p)}$ is a permissible face of $\beta^{(p+1)})$. Suppose further that $\alpha'\in[\alpha]$ and $\beta'\in[\beta]$. Then $\alpha'$ is a permissible face of $\beta'$.
\label{lem: permissibility holds for sym orbits}
\end{lem}

\begin{proof}
By Lemma \ref{lem: face relation holds for sym orbits}, we know $\alpha' < \beta'$. Suppose $\theta,\rho:[p] \to [p+1]$ are two injections that satisfy $\theta^*(\beta') = \alpha^*(\beta')=\alpha'$. Since $\alpha'$ and $\beta'$ are in the symmetric orbits of $\alpha$ and $\beta$, respectively, there are permutations $i:[p] \to [p]$ and $j:[p+1]\to[p+1]$ such that $i^*(\alpha')=\alpha$ and $j^*(\beta)=\beta'$. 
Now $j\circ\delta\circ i$ and $j\circ \alpha \circ i$ are two injections that glue $\alpha$ to $\beta$, so they have the same image. Then it must be true that $\im \theta = \im \rho$. 

Suppose $\phi:[p]\to[p+1]$ glues $\alpha$ to $\beta$; then $j\circ\phi\circ i$ glues $\alpha'$ to $\beta'$. Let $\pi \in \Aut(\beta')$. Define $f'$ by $f'(\pi) = (j\circ\phi\circ i)^{-1}|_{\im (j\circ\phi\circ i)} \circ \pi \circ (j\circ\phi\circ i)$. So $f' = c_i\circ f\circ c_j$, where $c_i$ and $c_j$ denote conjugation by $i$ and $j$, respectively. Since $f'$ is a composition of isomorphisms, it is an ismorphism.
\end{proof}

\noindent The following example shows permissible pairs in a symmetric $\Delta$-complex. 

\begin{example}
\label{example: permissible pairs on the half triangle}
Let $X$ be a symmetric $\Delta$-complex such that $|X|$ is the quotient of a standard 2-simplex by a reflection. The complex $X$ has a description as a symmetric $\Delta$-complex with two 0-simplices, three 1-simplices and three 2-simplices. See Figure \ref{fig: half triangle}. In this example, $[v]$ is a permissible face of $[a]$, $[w]$ of $[a]$, and $[c]$ of $[T]$. These are all the permissible pairs.
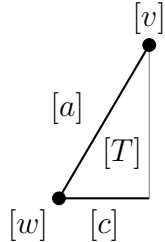
\begin{figure}[h]
    \centering
    \begin{tikzpicture}[scale=1.2]
        \draw[thick] (1,1.7) -- (0,0) -- (1,0);
        \draw[thin,gray] (1,1.7) -- (1,0);
        
        \draw[fill] (0,0) circle [radius=2pt];
        \node[below left] at (0,0) {$[w]$};
        
        \draw[fill] (1,1.7) circle [radius=2pt];
        \node[above] at (1,1.7) {$[v]$};
        
        \node[left] at (0.4,1) {$[a]$};
        \node[below] at (0.5,0) {$[c]$};
        \node at (0.7,0.5) {$[T]$};
    \end{tikzpicture}
    \caption{Half triangle. The gray edge is not a cell in this complex; it is a reference line.}
    \label{fig: half triangle}
\end{figure}
\end{example}

Permissibility in a symmetric $\Delta$-complex is similar to regularity in a CW-complex. In a CW-complex, if $\gamma^{(p-1)}$ is a regular face of $\alpha^{(p)}$ and $\alpha^{(p)}$ is a regular face of $\beta^{(p+1)}$, then there exists a $p$-cell $\alpha'\neq \alpha$ with $\gamma < \alpha' < \beta$. However, this statement fails for symmetric $\Delta$-complexes. Consider the unique symmetric $\Delta$-complex with one 0-simplex, one 1-simplex, and one 2-simplex. Its geometric realization is the quotient of a standard 2-simplex by the dihedral group $D_6$. There is only one 1-simplex between the vertex and the 2-simplex. The following lemma shows that permissibility prevents such undesirable situations. This lemma will be useful in the proofs of Lemma \ref{lem: DMF only one exception} and Lemma \ref{lem: level subcx only need to check one dim above}.

\begin{lem}
Let $X$ be a symmetric $\Delta$-complex, and let $0\leq r<p$. Suppose we have $\gamma^{(r)} < \alpha^{(p)} < \beta^{(p+1)}$. Assume further that $\alpha$ is a permissible face of $\beta$. Then there exists $\alpha'^{(p)}$ such that $[\alpha'] \neq [\alpha]$ with $\gamma < \alpha' < \beta$.
\label{lem: permissible parent must have another child}
\end{lem}

\begin{proof}
Let $f:[r] \to [p]$ glue $\gamma$ to $\alpha$ and $g:[p]\to [p+1]$ glue $\alpha$ to $\beta$. Let $l$ be the unique element in $[p+1]\setminus \im g$ and $k$ be the minimal element in the set $[p]\setminus \im f$. Define $g':[p]\to [p+1]$ as follows
\[g'(i) = \begin{cases}
g(i) & i \neq k \\
l & i=k
\end{cases}.\]
We have $g'\circ f = g\circ f$ since $k$ is not in the image of $f$. Let $\alpha' = (g')^*(\beta)$. We have $f^*(\alpha') = (g'\circ f)^*(\beta) = (g\circ f)^*(\beta) = \gamma$.

The $p$-simplex $\alpha'$ is not in the symmetric orbit of $\alpha$. Suppose it is; then there is a permutation $i:[p]\to[p]$ with $i^*(\alpha')=\alpha$. So $g$ and $g'\circ i$ both glue $\alpha$ to $\beta$ but have different images, which contradicts that $\alpha<\beta$ is permissible. So $[\alpha'] \neq [\alpha]$.
\end{proof}

\noindent Next, we give the definition of a symmetric CW-complex.

\begin{defn}\cite[Definition 2.1]{ACP-symCW}
A symmetric CW-complex is a Hausdorff topological space $X$ together with a partition into locally closed cells, such that, for each cell $C\subset X$, there is a continuous map from the quotient of the closed unit ball in $\R^p$ by a finite subgroup of the orthogonal group such that
\begin{enumerate}
    \item the quotient of the open unit ball maps homeomorphically onto $C$, and
    \item the image meets only finitely many cells, each of dimension less than $p$.
\end{enumerate}
A subset of $X$ is closed if and only if its intersection with closure of each cell is closed.
\label{defn: sym CW cx}
\end{defn}

\section{Discrete Morse Theory for symmetric $\Delta$-complexes}
\label{sec: dmt}
In this section, we develop a discrete Morse theory for symmetric $\Delta$-complexes. We give the definition of a discrete Morse function following \cite{Forman+discrete-morse} and prove Theorem \ref{thm: main theorem}. We also discuss the combinatorial criterion on matchings on the poset of symmetric orbits of symmetric $\Delta$-complexes following \cite{Chari} and prove Theorem \ref{thm: main thm combo}.
\subsection{Discrete Morse Theory through discrete Morse functions}
\begin{defn}
A discrete Morse function on a symmetric $\Delta$-complex $X$ is a function \[f:\coprod_{p\geq 0} X_p \to \R\] that satisfies the following conditions:
\begin{enumerate}[(i)]
    \item If $\alpha, \beta \in X_p$ are in the same symmetric orbit, then $f(\alpha)=f(\beta)$;
    \item If $\alpha^{(p)} < \beta^{(p+1)}$ is not permissible, then $f(\alpha) < f(\beta)$;
    \label{defn: exceptions only occur along permissible pairs}
    \item For any $\alpha^{(p)}$, \[u_f(\alpha^{(p)}) := \# \{[\beta^{(p+1)}]:\alpha^{(p)}<\beta^{(p+1)} \text{ is permissible and } f(\alpha) \geq f(\beta)\} \leq 1;\]
    \label{defn: quantity of beta}
    \item For any $\alpha^{(p)}$, \[d_f(\alpha^{(p)}):=\# \{[\gamma^{(p-1)}]:\gamma^{(p-1)}<\alpha^{(p)} \text{ is permissible and } f(\alpha) \leq f(\gamma)\}\leq 1.\]
    \label{defn: quantity of gamma}
\end{enumerate}
\label{defn: DMF}
\end{defn}
Note that in conditions \eqref{defn: quantity of beta} and \eqref{defn: quantity of gamma}, we are counting the number of symmetric orbits. This is well-defined by Lemma \ref{lem: permissibility holds for sym orbits}. Figure \ref{fig: dmf on half triangle} shows several discrete functions on the half triangle, only one of which is a discrete Morse function.

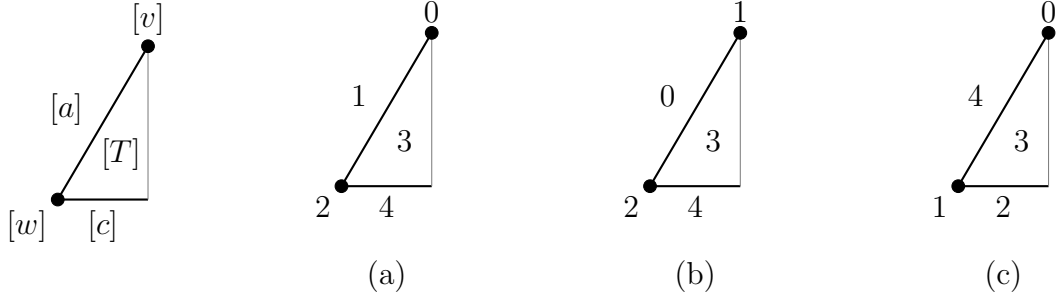
\begin{figure}[h]
    \centering
    \begin{minipage}{.24\textwidth}
    \begin{tikzpicture}[scale=1.2]
        \draw[thick] (1,1.7) -- (0,0) -- (1,0);
        \draw[thin,gray] (1,1.7) -- (1,0);
        
        \draw[fill] (0,0) circle [radius=2pt];
        \node[below left] at (0,0) {$[w]$};
        
        \draw[fill] (1,1.7) circle [radius=2pt];
        \node[above] at (1,1.7) {$[v]$};
        
        \node[left] at (0.4,1) {$[a]$};
        \node[below] at (0.5,0) {$[c]$};
        \node at (0.7,0.5) {$[T]$};
        
        \node at (0.5,-1) {};
    \end{tikzpicture}
    \end{minipage}
    \begin{minipage}{.24\textwidth}
    \begin{tikzpicture}[scale=1.2]
        \draw[thick] (1,1.7) -- (0,0) -- (1,0);
        \draw[thin,gray] (1,1.7) -- (1,0);
        
        \draw[fill] (0,0) circle [radius=2pt];
        \node[below left] at (0,0) {$2$};
        
        \draw[fill] (1,1.7) circle [radius=2pt];
        \node[above] at (1,1.7) {$0$};
        
        \node[left] at (0.4,1) {$1$};
        \node[below] at (0.5,0) {$4$};
        \node at (0.7,0.5) {$3$};
        
        \node at (0.5,-1) {(a)};
    \end{tikzpicture}
    \end{minipage}
    \begin{minipage}{.24\textwidth}
    \begin{tikzpicture}[scale=1.2]
        \draw[thick] (1,1.7) -- (0,0) -- (1,0);
        \draw[thin,gray] (1,1.7) -- (1,0);
        
        \draw[fill] (0,0) circle [radius=2pt];
        \node[below left] at (0,0) {$2$};
        
        \draw[fill] (1,1.7) circle [radius=2pt];
        \node[above] at (1,1.7) {$1$};
        
        \node[left] at (0.4,1) {$0$};
        \node[below] at (0.5,0) {$4$};
        \node at (0.7,0.5) {$3$};
        
        \node at (0.5,-1) {(b)};
    \end{tikzpicture}
    \end{minipage}
    \begin{minipage}{.24\textwidth}
    \begin{tikzpicture}[scale=1.2]
        \draw[thick] (1,1.7) -- (0,0) -- (1,0);
        \draw[thin,gray] (1,1.7) -- (1,0);
        
        \draw[fill] (0,0) circle [radius=2pt];
        \node[below left] at (0,0) {$1$};
        
        \draw[fill] (1,1.7) circle [radius=2pt];
        \node[above] at (1,1.7) {$0$};
        
        \node[left] at (0.4,1) {$4$};
        \node[below] at (0.5,0) {$2$};
        \node at (0.7,0.5) {$3$};
        
        \node at (0.5,-1) {(c)};
    \end{tikzpicture}
    \end{minipage}
    \caption{(a) is a discrete Morse function on the half triangle. (b) is not a discrete Morse function because $[a]$ has two faces with higher Morse values. (c) is not a discrete Morse function because $[a] < [T]$ is not permissible, so $[a]$ should have a smaller Morse value than $[T]$.}
    \label{fig: dmf on half triangle}
\end{figure}

Intuitively, a discrete Morse function assigns higher values to higher dimensional simplices while allowing for some exceptions. However, the conditions for a discrete Morse function are even more restrictive than they look, as shown by the following lemma.

\begin{lem}[Compare with Lemma 2.5 in \cite{Forman+discrete-morse}]
Let $f$ be a discrete Morse function on a symmetric $\Delta$-complex $X$. Let $p\geq 0$ and $\alpha \in X_p$. Then $u_f(\alpha)$ and $d_f(\alpha)$ cannot both attain 1.
\label{lem: DMF only one exception}
\end{lem}

\begin{proof}
If $p=0$ or $p=\dim X$, then the statement is true. Suppose $\gamma^{(p-1)}<\alpha^{(p)}<\beta^{(p+1)}$ are both permissible. Suppose we have $f(\gamma) \geq f(\alpha) \geq f(\beta)$. By Lemma \ref{lem: permissible parent must have another child}, there is another $p$-simplex $\alpha'$ such that $[\alpha']\neq [\alpha]$ and $\gamma < \alpha' < \beta$. By definition of a discrete Morse function, it must be true that $f(\gamma) < f(\alpha') < f(\beta)$. But that creates a contradiction:\[f(\gamma) < f(\alpha') < f(\beta) \leq f(\alpha) \leq f(\gamma).\] Contradiction!
\end{proof}

\begin{defn}
Let $X$ be a symmetric $\Delta$-complex and $f$ a discrete Morse function on it. A simplex $\alpha^{(p)}$ of $X$ is \textit{critical} if $u_f(\alpha^{(p)}) = 0$ and $d_f(\alpha^{(p)}) = 0.$
\label{defn: critical sx}
\end{defn}

Criticality is well-defined over symmetric orbits. In Figure \ref{fig: dmf on half triangle}(a), the vertex $[v]$ is the only critical simplex.

To prove the main theorem, we will take a similar approach as \cite{Forman+discrete-morse} by using elementary collapses along free faces. We shall define these concepts for symmetric $\Delta$-complexes.

\begin{defn}
Let $X$ be a symmetric $\Delta$-complex and $\alpha^{(p)}<\beta^{(p+1)}$ its simplices. We say $\alpha$ is a \textit{free face} of $\beta$ if there does not exist $\gamma^{(p+1)}$ with $[\gamma]\neq [\beta]$ such that $\alpha<\gamma$.
\label{defn: free face}
\end{defn}

Let $\Delta = \{\sum_{i=0}^p c_ie_i:\sum_{i=0}^p c_i = 1\}$ be the standard $p$-simplex where $\{e_0,\dots,e_p\}$ is the standard basis of $\R^{p+1}$. Let $\sigma \subset \Delta$ be the $(p-1)$-face of $\Delta$ defined by $c_p=0$. Recall that there is a strong deformation retraction from $\Delta$ to $\partial\Delta\setminus\{\sigma^\circ\}$ given by \[f:\Delta^p\times I \to \Delta^p, \quad (a,t) \mapsto a-tm_av,\] where $m_a = \min_{0\leq i\leq p-1}\{a_i\}$ and $v = (1,1,\dots,1,-p)$.

The following proposition is the main building block of the theory. We confirm that elementary collapse along a permissible face is a homotopy equivalence.
\begin{prop}
Let $X$ be a symmetric $\Delta$-complex and $Y\subset X$ be a subcomplex. Suppose $\gamma^{(p-1)} < \alpha^{(p)}$ is a free and permissible face. Suppose further that $(\bigcup_i X_i)\setminus (\bigcup_i Y_i) = [\gamma]\cup[\alpha]$. Then $|Y|$ is a strong deformation retract of $|X|$.
\label{prop: symmetric elementary collapse}
\end{prop}

\begin{proof}
Let $X_\alpha$ be the full subcategory of $X$ with $\ob(X_\alpha) = \{\sigma \in X_i: \sigma \leq \alpha\}$ and $Z_\alpha = Y\cap X_\alpha$. So $|X_\alpha|$ is the closure of the symmetric orbit of the open $p$-simplex $\alpha$. The subcomplex $|Z_\alpha|$ consists of all proper faces of $[\alpha]$ except for $[\gamma]$. Suppose $f:|X_\alpha| \times I \to |X_\alpha|$ is a strong deformation retraction to $|Z_\alpha|$. Then we may extend this map to be a strong deformation retraction $\bar{f}$ from $|X|$ to $|Y|$ by assigning $\bar{f}(x,t)=x$ for all $x\not\in |X_\alpha|$. Therefore, we may replace $X$ with $X_\alpha$ in our proof. Let $Z = Z_\alpha$.

We give a strong deformation retraction from $|X|$ to $|Z|$. Without loss of generality, we may assume that the inclusion $i_{p-1}:[p-1]\to[p]$ satisfies $i_{p-1}^*(\alpha)=\gamma$.

Since $[\alpha]$ is the unique maximal dimensional orbit of $X$, the continuous map determined by $\alpha \in X_p$, given by \[i_\alpha: \Delta^p \to |X|, \quad a \mapsto [(\alpha,a)],\] is surjective. Furthermore, $i_\alpha$ is an open mapping, because it is the composition of \[\Delta^p \xrightarrow{j_\alpha} \coprod_{i=0}^{p} \left(X_i \times \Delta^i\right) \xrightarrow{\pi} |X|,\] where $j_\alpha(a) = (\alpha,a)$, which is open, and $\pi$ is the canonical projection. Therefore $i_\alpha$ is a quotient map and there is a homeomorphism $\Delta^p/\sim_{i_\alpha} \cong |X|$, where $a\sim_{i_\alpha}b$ if and only if $i_\alpha(a)=i_\alpha(b)$. 

Recall that $\Aut(\alpha) \subset S_{p+1}$, so it acts on $\Delta^p$ by permuting coordinates and we may define the quotient space $\Delta^p/\Aut(\alpha)$. The map $i_\alpha$ induces a map from $\Delta^p/\Aut(\alpha) \to |X|$, and by abuse of notation, we also call the induced map $i_\alpha$. We will first show that the strong deformation retraction $f:\Delta^p\times I \to \Delta^p$ of a usual $p$-simplex along a $(p-1)$-face descends to a strong deformation retraction $f':\Delta^p/\Aut(\alpha) \times I \to \Delta^p/\Aut(\alpha)$, and then show that $f'$ descends further to $\bar{f}:|X| \times I \to |X|$, as shown in the diagram below.
\[\begin{tikzcd}
\Delta^p \times I \arrow[r,"f"] \ar[d,"q \times \id"] & \Delta^p \ar[d,"q"] \\
\Delta^p/\Aut(\alpha) \times I \arrow[r,dashrightarrow,"f'"] \ar[d,"i_\alpha \times \id"] & \Delta^p/\Aut(\alpha) \ar[d,"i_\alpha"] \\
\lvert X \rvert \times I \ar[r,dashrightarrow,"\bar{f}"] & \lvert X \rvert
\end{tikzcd}\]

To obtain $f'$, we check that $f$ is constant on $\Aut(\alpha)$-orbits of $\Delta^p\times I$, where $\Aut(\alpha)$ acts on $I$ trivially. If $\gamma<\alpha$ is permissible and $\gamma$ is glued to $\alpha$ by the inclusion $i_{p-1}:[p-1] \to [p]$, then it follows from item (1) in Definition \ref{defn: permissible face} that for any $\pi\in\Aut(\alpha)$, we have $\im \pi = \im d = [p-1]$, i.e., $\pi(p)=p$. So any two points $a$ and $b$ in the same orbit of $\Aut(\alpha)$ differ by a permutation of the first $p$ coordinates. Since $f(a,t) = a-tm_av$, where $m_a$ is the value of the entry that is minimal among the first $p$ coordinates of $a$, we have $m_a = m_b$. Therefore, the function $f$ descends to a function on $\Delta^p/\Aut(\alpha)$.

Next, we check that $f'$ descends to $\bar{f}$. Suppose $\bar{a},\bar{b}\in\Delta^p/\Aut(\alpha)$ are distinct elements satisfying that $i_\alpha(\bar{a}) = i_\alpha(\bar{b})$. We claim that both $a$ and $b$ must be contained in the set $\{x\in \Delta^p: c_i=0 \text{ for some }i\neq p\}$. Suppose not. By assumption,
\begin{align*}
    i_\alpha(\bar{a}) &= i_\alpha(\bar{b}) \\
    [(\alpha,a)] &= [(\alpha,b)].
\end{align*} By unpacking the definition of $|X|$, we see that if $(\sigma,r) \sim (\tau,s)$, then $r$ and $s$ must have the same number of nonzero coordinates. Therefore, if one of $a$ and $b$ is in the interior of $\Delta^p$, then the other must be too. Thus $(\alpha,a)$ and $(\alpha,b)$ must be identified via an automorphism of $\alpha$, which contradicts that $\bar{a} \neq \bar{b}$. Now suppose $a$ has last coordinate zero and all others nonzero. That $i_\alpha(a) = i_\alpha(b)$ implies that there exist $(\sigma^{(p-1)},c)$ and $\theta,\phi:[p-1]\to[p]$ such that $(\alpha,a) \sim (\sigma^{(p-1)},c) = (\theta^*\alpha,c)$ and $(\alpha,b) \sim (\sigma^{(p-1)},c) = (\phi^*\alpha,c)$ in $|X|$. As $a$ has last coordinate zero, it must be true that $\im \theta = [p-1]$. So the simplex $\sigma \in [\gamma]$, but that means $\sigma<\alpha$ is permissible, and thus $\im \phi = [p-1]$ by item (1) of Definition \ref{defn: permissible face}. Therefore the point $b$ must be in the interior of the facet $c_p=0$ as well. But that means these two points can only be identified via some $\psi \in \Aut(\sigma)$ in $|X|$. But by item (2) of Definition \ref{defn: permissible face}, every automorphism of $\sigma$ is induced by an automorphism of $\alpha$, so it must be true that $\bar{a} = \bar{b} \in \Delta^p/\Aut(\alpha)$, which is a contradiction. If both $a$ and $b$ are in $\{x\in \Delta^p: c_i=0 \text{ for some }i\neq p\}$, then $f'(\bar{a},t) = \bar{a} = f'(\bar{b},t) = \bar{b}$. So $f'$ descends to $\bar{f}$ as desired.

The map $\bar{f}$ is continuous by the universal property of quotient spaces. The image of $\bar{f}$ is $\im(f\circ i_\alpha) = |Z|$ and one can verify that $\bar{f}(z,t)=z$ for $z\in |Z|\subset |X|$ for all $t$. Therefore $\bar{f}$ is a strong deformation retraction from $|X|$ to $|Z|$.
\end{proof}

Proposition \ref{prop: symmetric elementary collapse} shows that $|Y|$ is homotopy equivalent to $|X|$. We call such a reduction an \textit{elementary collapse} for a symmetric $\Delta$-complex. See Figure \ref{fig: elem collapse} for an example of an elementary collapse.

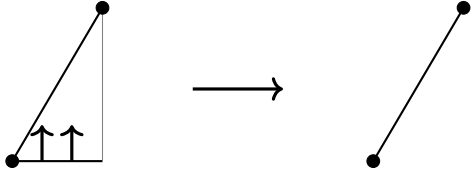
\begin{figure}[h]
    \centering
    \begin{tikzpicture}[scale=1.2]
        \draw[thick] (1,1.7) -- (0,0) -- (1,0);
        \draw[thin,gray] (1,1.7) -- (1,0);
        
        \draw[fill] (0,0) circle [radius=2pt];
        \draw[fill] (1,1.7) circle [radius=2pt];
        
        \draw[very thick,->] (0.33,0) -- (0.33,0.4);
        \draw[very thick,->] (0.66,0) -- (0.66,0.4);
        
        \draw[very thick,->] (2,0.8) -- (3,0.8);
        
        \draw[thick] (5,1.7) -- (4,0);
        \draw[fill] (4,0) circle [radius=2pt];
        \draw[fill] (5,1.7) circle [radius=2pt];
    \end{tikzpicture}
    \caption{Collapsing the 2-simplex in the half triangle along the bottom half edge is an elementary collapse.}
    \label{fig: elem collapse}
\end{figure}

\begin{defn}
Let $X$ be a symmetric $\Delta$-complex and $f$ a discrete Morse function on it. Let $c\in \R$. The level subcomplex $K(c)$ is the full subcategory of $X$ defined by \[\ob(K(c)) = \bigcup_{\substack{\alpha\in X_p \\ f(\alpha) \leq c}}\bigcup_{\beta < \alpha}\beta.\]
\end{defn}

\noindent In other words, $K(c)$ is the subcomplex of $X$ generated by simplices with discrete Morse value at most $c$. Figure \ref{fig: level subcx} shows different level subcomplexes of a given discrete Morse function.

\begin{figure}[h]
    \centering
    \begin{minipage}{0.23\textwidth}
    \begin{tikzpicture}[scale=1.2]
        \draw[thick] (1,1.7) -- (0,0) -- (1,0);
        \draw[thin,gray] (1,1.7) -- (1,0);
        
        \draw[fill] (0,0) circle [radius=2pt];
        \node[below left] at (0,0) {$2$};
        
        \draw[fill] (1,1.7) circle [radius=2pt];
        \node[above] at (1,1.7) {$0$};
        
        \node[left] at (0.4,1) {$1$};
        \node[below] at (0.5,0) {$4$};
        \node at (0.7,0.5) {$3$};
        
        \node at (0.5,-1) {};
    \end{tikzpicture}
    \end{minipage}
    \begin{minipage}{0.23\textwidth}
    \begin{tikzpicture}[scale=1.2]
        \draw[fill] (1,1.7) circle [radius=2pt];
        \node[above] at (1,1.7) {$0$};
        \node at (0,0) {};
        
        \node at (0.5,-1) {$K(0)$};
    \end{tikzpicture}
    \end{minipage}
    \begin{minipage}{0.23\textwidth}
    \begin{tikzpicture}[scale=1.2]
        \draw[thick] (1,1.7) -- (0,0);
        
        \draw[fill] (0,0) circle [radius=2pt];
        \node[below left] at (0,0) {$2$};
        
        \draw[fill] (1,1.7) circle [radius=2pt];
        \node[above] at (1,1.7) {$0$};
        
        \node[left] at (0.4,1) {$1$};
        
        \node at (0.5,-1) {$K(1)=K(2)$};
    \end{tikzpicture}
    \end{minipage}
    \begin{minipage}{0.23\textwidth}
    \begin{tikzpicture}[scale=1.2]
        \draw[thick] (1,1.7) -- (0,0) -- (1,0);
        \draw[thin,gray] (1,1.7) -- (1,0);
        
        \draw[fill] (0,0) circle [radius=2pt];
        \node[below left] at (0,0) {$2$};
        
        \draw[fill] (1,1.7) circle [radius=2pt];
        \node[above] at (1,1.7) {$0$};
        
        \node[left] at (0.4,1) {$1$};
        \node[below] at (0.5,0) {$4$};
        \node at (0.7,0.5) {$3$};
        
        \node at (0.5,-1) {$K(3)=K(4)$};
    \end{tikzpicture}
    \end{minipage}
    \caption{Given a discrete Morse function on the half triangle, as in the leftmost picture, we construct the level subcomplexes.}
    \label{fig: level subcx}
\end{figure}
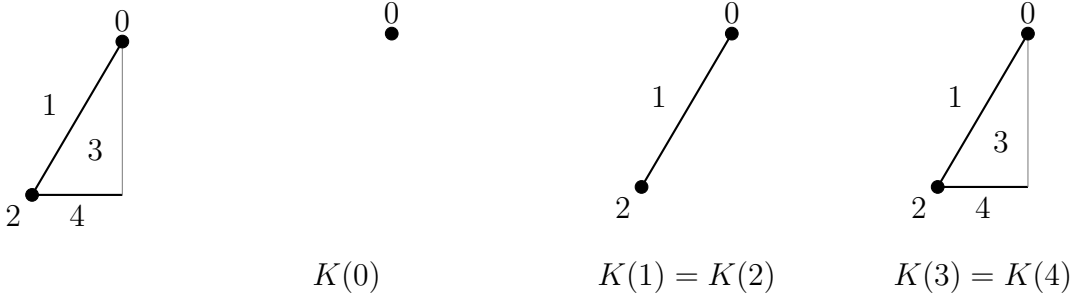

In general, to check if a simplex $\sigma$ with $f(\sigma) > c$ is in $K(c)$, we need to check if there is $\tau$ with $\sigma < \tau$ and $f(\tau) \leq c$. The next lemma shows that it is enough to check such $\tau$ with $\dim \tau = \dim \sigma + 1$.

\begin{lem}[Compare with Lemma 3.2 in \cite{Forman+discrete-morse}]
Let $X$ and $f$ be as before. Let $\sigma$ be a $p$-simplex and $\tau^{(k)} > \sigma^{(p)}$. Then there is a $(p+1)$-simplex $\tau'$ with $\sigma < \tau'$ and $f(\tau') \leq f(\tau)$.
\label{lem: level subcx only need to check one dim above}
\end{lem}

\begin{proof}
If $k = p+1$, then we let $\tau'=\tau$. Otherwise, let $\nu^{(k-1)}$ be such that $\sigma < \nu < \tau$. If $\nu$ is not a permissible face of $\tau$, then $f(\nu) < f(\tau)$. If it is, then by Lemma \ref{lem: permissible parent must have another child}, there exists $\nu'^{(k-1)}$ such that $[\nu']\neq [\nu]$ and $\sigma < \nu' < \tau$. By definition of $f$, either $f(\nu) < f(\tau)$ or $f(\nu') < f(\tau)$. We obtain the lemma by induction.
\end{proof}

\begin{thm}[Compare with Theorem 3.3 in \cite{Forman+discrete-morse}]
If there are no critical orbits $[\alpha]$ with $f(\alpha) \in (a,b]$, then $|K(b)|$ is homotopy equivalent to $|K(a)|$.
\label{thm: level subcx with no critical orbit}
\end{thm}

\begin{proof}
We may perturb $f$ slightly so it is injective without changing $K(a)$, $K(b)$, or which simplices are critical.

If $f^{-1}((a,b]) = \emptyset$, then $K(b) = K(a)$. Otherwise, by the injectivity of $f$, we may assume there is exactly one symmetric orbit $[\sigma^{(p)}]$ with $f([\sigma]) \in (a,b]$. Since $[\sigma]$ is not critical, it satisfies exactly one of
\begin{enumerate}
    \item there exists $[\beta^{(p+1)}]>[\sigma]$ with $f(\beta) < f(\sigma)$;
    \item there exists $[\alpha^{(p-1)}] < [\sigma]$ with $f(\alpha) > f(\sigma)$.
\end{enumerate}
In case (1), we must have $f(\beta) \leq a$, so $\beta \in K(a)$. Since $\sigma < \beta$, we have $\sigma \in K(a)$. Therefore, $K(b) = K(a)$.

In case (2), we will show $\ob(K(b)) = \ob(K(a)) \sqcup [\sigma] \sqcup [\alpha]$. By Lemma \ref{lem: DMF only one exception}, all $\tau^{(p+1)} > \sigma$ satisfy $f(\tau) > f(\sigma)$, so $f(\tau) > b$. Therefore by Lemma \ref{lem: level subcx only need to check one dim above}, we have $\sigma \not \in K(a)$. Now if there is another $(p-1)$-simplex $\alpha'$ such that $[\alpha'] \neq [\alpha]$ and $\alpha' < \sigma$, then $f(\alpha') < f(\sigma)$, so $f(\alpha') < a$. Hence $\alpha'$ and all its faces are contained in $K(a)$. If $\sigma'$ is another $p$-simplex such that $[\sigma'] \neq [\sigma]$ and $\alpha < \sigma'$, then $f(\sigma') > f(\alpha) > b$, which means $\alpha \not \in K(a)$. Therefore, by Proposition \ref{prop: symmetric elementary collapse}, the geometric realization $|K(a)|$ is a strong deformation retract of $|K(b)|$.
\end{proof}

\noindent Let $B^p \subset \R^p$ be the closed unit ball. Let $O(p)$ be the orthogonal group of dimension $p$.

\begin{thm}[Compare with Theorem 3.4 in \cite{Forman+discrete-morse}]
Suppose $[\sigma]$ is the symmetric orbit of a critical $p$-simplex with $f(\sigma) \in (a,b]$ and $f^{-1}((a,b])$ contains no other critical orbits. Then there is a continuous map $F:\partial B^p/G_\sigma \to |K(a)|$, where $G_\sigma$ is a finite subgroup of $O(p)$ such that $|K(b)|$ is homotopy equivalent to $|K(a)|\cup_F B^p/G_\sigma$.
\label{thm: level subcx with critical orbit}
\end{thm}

\begin{proof}
As before, we may assume that $f$ is injective and $[\sigma]$ is the only symmetric orbit with $f(\sigma) \in (a,b]$. Since $\sigma$ is critical, for any $\tau^{(p+1)} > \sigma$, we have $f(\tau) > f(\sigma)$. So $f(\tau) > b$ for all such $\tau$. Lemma \ref{lem: level subcx only need to check one dim above} implies that $[\sigma] \cap K(a) = \emptyset$. On the other hand, for any $\nu^{(p-1)} < \sigma$, we have $f(\nu) < f(\sigma)$, so $f(\nu) < a$. That means all faces of $\sigma$ are contained in $K(a)$. So $\sigma$ defines a continuous map $i_\sigma:\Delta^p/\aut(\sigma) \to |K(b)|$. The map $i_\sigma$ satisfies that $i_\sigma|_{\partial \Delta^p/\aut(\sigma)}:\partial \Delta^p/\aut(\sigma) \to |K(a)|$ and that it is a homeomorphism on the quotient of the interior of $\Delta^p$. 

We identify the hyperplane in $\R^{p+1}$ containing $\Delta^p$ with $\R^p$ and by identifying the center of $\Delta^p$ with the origin of $\R^p$, we may extend the action of $S_{p+1}$ to $\R^p$. The symmetric group $S_{p+1}$ is generated by transpositions $(ij)$, where $i,j \in \{0,\dots,p\}$ and $i\neq j$. A transposition $(ij)$ acts on $\Delta^p$ and hence $\R^p$ by reflection, which can be identified with an element of $O(p)$. Therefore there is a natural identification between $\aut(\sigma)$ and a finite subgroup $G_\sigma$ of $O(p)$. Now $i_\sigma$ defines a continuous map $\Delta^p/G_\sigma \to |K(b)|$ where $\Delta^p$ is viewed as a subset of $\R^p$. By choosing a homotopy between $\Delta^p$ and $B^p$, we obtain $F:B^p/G_\sigma \to |K(b)|$ such that $F$ is a homotopy of $i_\sigma$ and $F$ satisfies that it maps $\partial B^p/G_\sigma$ to $|K(a)|$ and that it is a homeomorphism on the quotient of the interior of $B^p$. Thus $|K(b)| \simeq |K(a)|\cup_F B^p/G_\sigma$.
\end{proof}

\begin{proof}[Proof (of Theorem \ref{thm: main theorem})]
Let $X$ be a finite nonempty symmetric $\Delta$-complex with a discrete Morse function $f$. We may assume $f$ is injective on symmetric orbits as before. Since $X$ is finite, there exits finite values $c_0 = \min_{\sigma \in \cup X_p}\{f(\sigma)\}$ and $c_n = \max_{\sigma \in \cup X_p}\{f(\sigma)\}$. We claim that $|K(c_0)|$ is a point. Since $X$ is nonempty and $f$ is injective on symmetric orbits, there is exactly one symmetric orbit $[\sigma]$ with $f(\sigma) = c_0$. If $\sigma$ has dimension greater than 0, then it has at least two codimension-1 faces $\tau$ and $\tau'$ with $f(\tau) > f(\sigma)$ and $f(\tau') > f(\sigma)$. If $[\tau] = [\tau']$, then $(\tau,\sigma)$ is not permissible, so item (ii) of Definition \ref{defn: DMF} is violated. If $[\tau] \neq [\tau']$, then item (iv) of Definition \ref{defn: DMF} is violated. Therefore $\sigma$ is a 0-simplex, hence $|K(c_0)|$ is a single point.

Choose $c_0 < c_1 < \cdots < c_n$ so that there is exactly one symmetric orbit whose discrete Morse value is in $(c_i,c_{i+1})$ for $i=0,\dots,n-1$. By induction, we assume $|K(c_i)|$ is a symmetric CW-complex. By Theorems \ref{thm: level subcx with no critical orbit} and \ref{thm: level subcx with critical orbit}, either $|K(c_{i+1})| \simeq |K(c_i)|$ or $|K(c_{i+1})| \simeq |K(c_i)| \cup_F B^p/G$ for some finite subgroup $G$ of $O(p)$. In the latter case, we declare the image of $B^p/G$ in $|K(c_i)| \cup_F B^p/G$ to be a cell. Therefore $|K(c_n)| = |X|$ is a symmetric CW-complex with exactly one $p$-cell $(B^p)^\circ/G$ per critical symmetric orbit $[\alpha]$ of dimension $p$ with $G \cong \aut(\alpha)$.
\end{proof}

Figure \ref{fig: dmf collapse} shows how we can use a discrete Morse function to collapse simplices while maintaining the homotopy type of the complex.

\begin{figure}[h]
    \centering
    \begin{minipage}{0.23\textwidth}
    \begin{tikzpicture}[scale=1.2,baseline=-2cm]
        \draw[thick] (1,1.7) -- (0,0) -- (1,0);
        \draw[thin,gray] (1,1.7) -- (1,0);
        
        \draw[fill] (0,0) circle [radius=2pt];
        \node[below left] at (0,0) {$2$};
        
        \draw[fill] (1,1.7) circle [radius=2pt];
        \node[above] at (1,1.7) {$0$};
        
        \node[left] at (0.4,1) {$1$};
        \node[below] at (0.5,0) {$4$};
        \node at (0.7,0.5) {$3$};
        
        \node at (0.5,-1) {};
    \end{tikzpicture}
    \end{minipage}
    \begin{minipage}{0.23\textwidth}
    \begin{tikzpicture}[scale=1.2]
        \draw[thick] (1,1.7) -- (0,0) -- (1,0);
        \draw[thin,gray] (1,1.7) -- (1,0);
        \draw[fill] (0,0) circle [radius=2pt];
        \draw[fill] (1,1.7) circle [radius=2pt];
        
        \draw[very thick,->] (0.33,0) -- (0.33,0.4);
        \draw[very thick,->] (0.66,0) -- (0.66,0.4);
        
        \draw[very thick,->] (2,1) -- (3,1);
        
        \node at (0.5,-1) {$K(4) \searrow K(2)$};
    \end{tikzpicture}
    \end{minipage}
    \begin{minipage}{0.23\textwidth}
    \begin{tikzpicture}[scale=1.2]
        \draw[thick] (1,1.7) -- (0,0);
        
        \draw[fill] (0,0) circle [radius=2pt];
        \draw[fill] (1,1.7) circle [radius=2pt];
        
        \draw[very thick,->] (0,0) -- (0.3, 0.51);
        
        \draw[very thick,->] (2,1) -- (3,1);
        
        \node at (0.5,-1) {$K(2)\searrow K(0)$};
    \end{tikzpicture}
    \end{minipage}
    \begin{minipage}{0.23\textwidth}
    \begin{tikzpicture}[scale=1.2]
        \node at (0,0) {};
        
        \draw[fill] (1,1.7) circle [radius=2pt];
        
        \node at (0.5,-1) {$K(0)$};
    \end{tikzpicture}
    \end{minipage}
    \caption{The discrete Morse function on the left dictates a sequence of elementary collapses. The direction of deformation retraction is indicated by arrows at each step.}
    \label{fig: dmf collapse}
\end{figure}
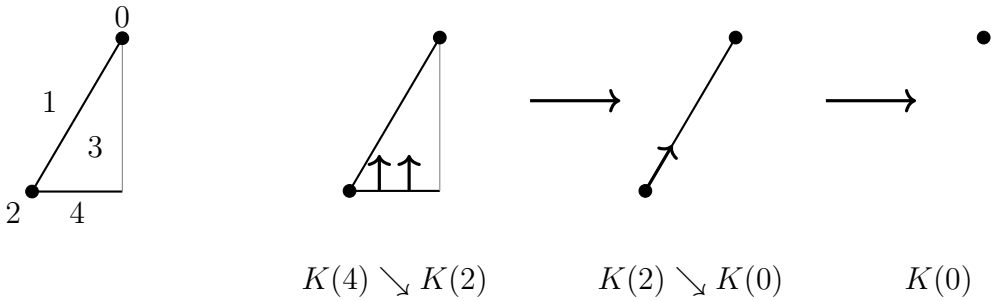

\begin{rmk}
A usual simplicial complex $C$ can be viewed as the geometric realization of a functor from $I^{\mathrm{op}}_{\mathrm{ord}} \to \text{Sets}$, where $I_{\mathrm{ord}}$ is the category of finite sets and order-preserving injections. For convenience we call this functor $C$ as well. There is a symmetric $\Delta$-complex $X_C$ whose geometric realization coincides with $C$ with the following description. On objects, $X_C([k]) = \coprod_{\sigma\in S_{k+1}} C([k])$. On morphisms, a permutation $\rho:[k] \to [k]$ is mapped to the function defined by $X_C(\rho)(\sigma,a) = (\rho\circ \sigma,a)$. The inclusion $i_k:[k] \to [k+1]$ is mapped to the function defined by $X_C(i_k)(\sigma,a) = (\pi,b)$ where $\pi$ and $b$ are described as follows. Let $i = \sigma \circ i_k$. Since $i$ is an injection, $\im i = [k+1]\setminus \{m\}$ for some $m$. Now there is exactly one order-preserving injection $h_m:[k] \to [k+1]$ whose image misses $m$. We define $\pi = h_m^{-1}|_{\im i} \circ i$ and $b = C(h_m)(a)$.

With this description, every pair $\alpha^{(p)}<\beta^{(p+1)}$ of simplices in $X_C$ is permissible, so there is a one-to-one correspondence between discrete Morse functions on $X_C$ using Definition \ref{defn: DMF} and usual discrete Morse functions on $C$. Furthermore, the collapses prescribed by our discrete Morse function on $X_C$ agree with the collapses given by the corresponding usual discrete Morse function on $C$.
\end{rmk}

\subsection{Discrete Morse Theory with a combinatorial description}
\label{sec: dmt combo description}

In this section, we give a combinatorial description of discrete Morse theory for symmetric $\Delta$-complexes. We start with some definitions from poset theory. Let $(P,<)$ be a poset. Let $\prec$ denote covering relations in $P$. In other words, $a \prec b \in P$ if there is no element $c\in P$ with $a < c < b$.

\begin{defn}\cite[Definition 11.1]{Kozlov}
A partial matching on a poset $P$ is a partial matching in the underlying graph of the Hasse diagram of $P$. It is a subset $M\subset P \times P$ such that
\begin{itemize}
    \item $(a,b) \in M$ implies that $a\prec b$;
    \item each element of $P$ belongs to at most one element in $M$.
\end{itemize}
If $(a,b) \in M$, we write $a = d(b)$ and $b=u(a)$, where $d$ stands for down and $u$ stands for up.

$M$ is acyclic if there does not exist a cycle \[b_1 \succ d(b_1) \prec b_2 \succ d(b_2) \prec \cdots \prec b_n \succ d(b_n) \prec b_1,\] with $n\geq 2$ and all $b_i\in P$ distinct.
\end{defn}

\begin{rmk}
The condition that a partial matching $M$ on a poset $P$ is acyclic can be interpreted intuitively as follows. Let $G_P$ the Hasse diagram of $P$ regarded as a graph. Orient the edges of $G_P$ from larger elements to smaller ones. Now given a partial matching $M$, reverse the orientations of the edges in $G_P$ if they are in $M$. The matching $M$ is acyclic if and only if the directed graph $G_P$ contains no directed cycles.
\label{rmk: acyclic matching in directed graph}
\end{rmk}

Let $X$ be a symmetric $\Delta$-complex. Let $H_X$ be the poset of symmetric orbits of $X$ with face relations. A partial matching $M$ on $H_X$ is called permissible if every element $([\alpha^{(p)}],[\beta^{(p+1)}]) \in M$ satisfy that $[\alpha] < [\beta]$ is permissible. We prove theorem \ref{thm: main thm combo}.

\begin{proof}[Proof of Theorem \ref{thm: main thm combo}]
It suffices to prove that there is a discrete Morse function $f$ on $X$ such that for each $([\alpha^{(p)}],[\beta^{(p+1)}]) \in M$, we have $f(\alpha^{(p)}) \geq f(\beta^{(p+1)})$. And the only critical symmetric orbits of $f$ are the ones that do not appear in $M$.

Let $G(M)$ be the directed graph as constructed in Remark \ref{rmk: acyclic matching in directed graph}. By \cite[Theorem 3.6]{Forman+users-guide}, there is a real-valued function $g$ of the vertices that is strictly decreasing along each directed path. This is the discrete Morse function that we seek.
\end{proof}

\begin{rmk}
A discrete Morse function $f$ gives a permissble acyclic matching $M$ on $H_X$ by setting \[M = \{([\alpha^{(p)}],[\beta^{(p+1)}]):\alpha < \beta, f(\alpha) \geq f(\beta), p\geq 0\}.\]
\end{rmk}

\section{Application to the moduli space of tropical abelian varieties}
\label{sec: Ag}
In this section, we apply our generalized discrete Morse theory to the link of the origin $LA^{\mathrm{trop},\mathrm{P}}_g$ in the moduli space of principally polarized tropical abelian varieties of dimension $g$ with respect to the perfect cone decomposition. This space is interesting because of its connection to the algebraic moduli stack $\mathcal{A}_g$ of principally polarized abelian varieties of dimension $g$. More specifically, the rational homology of the link is identified with the top-weight cohomology $\mathcal{A}_g$ \cite[Theorem 3.1]{brandt2020top}.

\subsection{Background on $A^{\mathrm{trop},\mathrm{P}}_g$}

A rational polyhedral cone $\sigma \subset \R^g$ is the non-negative real span of vector $v_1,\dots,v_n \in \Z^g$ \[\sigma := \R_{\geq 0}\langle v_1,\dots,v_n \rangle := \left\{ \sum_{i=1}^n \lambda_iv_i:\lambda \in \R, \lambda \geq 0\right\}.\] In this section all cones $\sigma \subset \R^g$ are assumed to be strong convex, i.e., $\sigma$ does not contain any nonzero linear subspaces of $\R^g$. The dimension of $\sigma$ is the dimension of its linear span. A face of $\sigma$ is a nonempty subset of $\sigma$ that minimizes a linear functional on $\R^g$, and it is a rational polyhedral cone itself. The one-dimensional faces of $\sigma$ are called extremal rays of $\sigma$. There are only finitely many of them. An automorphism of a strongly convex rational polyhedral cone permutes its extremal rays and is uniquely determined by this permutation.

Let $\R^{\binom{g+1}{2}}$ be the vector space of quadratic forms in $\R^g$ represented by $g \times g$ real symmetric matrices. Let $\Omega_g^\mathrm{rt}$ be the cone in $\R^{\binom{g+1}{2}}$ consisting of positive semidefinite quadratic forms whose kernel is defined over $\Q$. The group $\gl_g(\Z)$ acts on a quadratic form $Q$ by $h\cdot Q = hQh^t$, where $h\in \gl_g(\Z)$ and $h^t$ is its transpose. The cone $\Omega_g^\mathrm{rt}$ is preserved under the $\gl_g(\Z)$-action.

We recall the perfect cone decomposition of $\Omega_g^\mathrm{rt}$. It is an example of an admissible decomposition, see \cite[Definition 2.5]{brandt2020top}. Let $Q$ be a positive definite quadratic form. Define \[M(Q) := \{v \in \Z^g\backslash \{0\}:Q(v) \leq Q(v'), \forall v' \in \Z^g\backslash \{0\}\}.\] Define the cone associated to $Q$ by \[\sigma[Q] := \R_{\geq 0}\langle v\cdot v^t \rangle_{v\in M(Q)}.\] The rank of $\sigma[Q]$ is the maximum rank of an element of $\sigma[Q]$ and eqivalently the dimension of the span of $M(Q)$.

\begin{fct}\cite{voronoi}
The set of cones \[\Sigma_g^\mathrm{P} := \{\sigma[Q]:Q \text{ a positive definite form on }\R^g\}\] is an admissible decomposition of $\Omega_g^\mathrm{rt}$. It is called the perfect cone decomposition.
\end{fct}

\begin{defn}\cite{brandt2020top}
A principally polarized tropical abelian variety, or tropical abelian variety for short, of dimension $g$ is $A = (\R^g/\Z^g,Q)$, where $Q$ is a positive semidefinite symmetric bilinear form on $\R^g$ with rational null space.
\end{defn}

Two tropical abelian varieties $(\R^g/\Z^g,Q)$ and $(\R^g/\Z^g,Q')$ are isomorphic if there is $h\in \gl_g(\Z)$ such that $Q'=hQh^t$.

\begin{defn}\cite{brandt2020top}
The moduli space of tropical abelian varieties of dimension $g$ with respect to an admissible decomposition $\Sigma$ of $\Omega^\mathrm{rt}_g$ is a generalized cone complex \[A^{\mathrm{trop},\Sigma}_g = \colim\{\sigma\}_{\sigma\in\Sigma}\] with arrows given by inclusion of faces composed with the action of $\gl_g(\Z)$ on $\Omega^\mathrm{rt}_g$.
\end{defn}

We denote by $A^{\mathrm{trop},\mathrm{P}}_g$ the tropical abelina varieties of dimension $g$ with respect to the perfect cone decomposition. Next, we define matroidal cones and cones with coloops.

\begin{defn}
A matroid $M=(E,\mathcal{C})$ on a finite set $E$ is a subset $\mathcal{C} \subset \mathcal{P}(E)\setminus \{\emptyset\}$, called the set of circuits of $M$, that satisfies the following axioms:

(C1) No proper subset of a circuit is a circuit.

(C2) If $C_1,C_2 \in \mathcal{C}$ are distinct and $c\in C_1\cap C_2$, then $(C_1\cup C_2)\setminus \{c\}$ contains a circuit.
\end{defn}

A matroid $M=(E,\mathcal{C})$ is simple if it has no circuits of length 1 or 2. It is representable over a field $k$ if there is a matrix $A$ over $k$ such that $E$ indexes the columns of $A$ and elements of $\mathcal{C}$ are exactly the minimal linearly dependent sets of columns. The matroid $M$ is regular if it is representable over every field.

A matrix is totally unimodular if every minor is -1,0, or 1. The rank of a regular matroid $M$ is the smallest number $r$ such that $M$ is represented by an $r \times n$ totally unimodular matrix over $\R$ for some $n$.

Given a simple, regular matroid $M$ of rank $\leq g$, we may choose a $g\times n$ totally unimodular matrix $A$ that represents it. Let the columns of $A$ be $v_1,\dots,v_n$. Define the cone associated to $A$ by \[\sigma_A(M) := \R_{\geq 0} \langle v_1v_1^t,v_2v_2^t,\dots,v_nv_n^t \rangle.\] The cone $\sigma_A(M)$ is a perfect cone in $\Sigma_g^\mathrm{P}$, see \cite[Theorem 4.2.1]{melo-viviani-perfect-decomp}. Further more, if two matrices $A$ and $A'$ both represent $M$, then $\sigma_A(M)$ and $\sigma_{A'}(M)$ are related by a $\gl_g(\Z)$-action, so we denote the $\gl_g(\Z)$-orbit of $\sigma_A(M)$ by $\sigma(M)$. An $n$-dimensional cone is called a matroidal cone if it arises this way.

\begin{defn}
The regular matroid cone complex $A_g^\mathrm{R}$ in $A^{\mathrm{trop},\mathrm{P}}_g$ is the union of all matroidal cones.
\end{defn}

The regular matroid cone complex is a subcomplex because the faces of a matroidal cone are matroidal cones. Matroidal cones are simplicial.

We further define the coloop cone subcomplex of $A_g^\mathrm{R}$. An element $e$ of a matroid $M$ is called a coloop if it does not belong to any circuit of $M$. When $M$ is regular, this means there is a totally unimodular matrix $A = [v_1, v_2,\dots,v_n]$ such that the $n$th entry of each column $(v_i)_n = 0$ for $i=1,\dots,n-1$ and $v_n = (0,\dots,0,1)$ that represents $M$, where the element $e$ indexes the last column.

Recall that the rank of a perfect cone $\sigma[Q]$ is defined to be the maximum rank of an element in $\sigma[Q]$.

\begin{lem}
Let $\sigma(M)$ be a matroidal cone. Then the rank of $\sigma(M)$ is equal to the rank of $M$.
\end{lem}

\begin{proof}
Let $M$ be a simple, regular matroid on the ground set $\{1,\dots,n\}$. Suppose $M$ has rank $r$. That means there is a totally unimodular matrix $A$ of dimension $r\times n$ representing $M$. Note that $A$ must be full-rank and have $r \leq n$ because otherwise $r$ is not minimum among matrices representing $M$. We may rearrange the columns of $A$ so that the rightmost $r \times r$ submatrix is full-rank; call this submatrix $B$. Since $A$ is totally unimodular, the matrix $B \in \gl_r(\Z)$. So $B^{-1}A$ still represents $M$, and the rightmost $r\times r$ submatrix is the identity. Using $B^{-1}A$, the cone $\sigma(M)$ has the following description: \[\sigma(M) = \R_{\geq 0}\langle v_1v_1^t,\dots,v_{n-r}v_{n-r}^t,e_1e_1^t,\dots,e_re_r^t\rangle ,\] where $e_i$ is the vector that has 1 in the $i$th entry and 0 elsewhere. One element in $\sigma(M)$ achieves the largest possible rank, namely $\Sigma_{i=1}^r e_ie_i^t = I_r$. Therefore the rank of $\sigma(M)$ is also $r$.
\end{proof}

\begin{defn}
The coloop locus $A_g^\mathrm{C}$ in $A_g^\mathrm{R}$ is the union of matroidal cones $\sigma$ such that either:
\begin{enumerate}[(1)]
    \item the rank of $\sigma$ is less than $g$, or
    \item the rank of $\sigma$ is equal to $g$ and $\sigma$ has one coloop.
\end{enumerate}
\end{defn}

\noindent The locus $A_g^\mathrm{C}$ is a subcomplex.

Let $\Sigma_{g}^\mathrm{P,R}$ be the collection of all matroidal perfect cones. Following \cite[Definition 5.7]{brandt2020top}, the set $(\Sigma_{g,\mathrm{nco}}^\mathrm{P}[n] \cap \Sigma_{g}^\mathrm{P,R}[n])/\gl_g(\Z)$ is the $\gl_g(\Z)$-orbits of $n+1$-dimensional matroid cones of rank $< g$ with no coloops. The inflation map defined in \cite[Definition 5.8]{brandt2020top} restricts to $(\Sigma_{g,\mathrm{nco}}^\mathrm{P}[n] \cap \Sigma_{g}^\mathrm{P,R}[n])/\gl_g(\Z)$. Given an element in it, choose a representative \[\sigma(M) = \R_{\geq 0}\langle v_1v_1^t,\dots,v_nv_n^t \rangle\] such that the $g$th row is all zero. Then we have \[\ifl(\sigma(M)) = \R_{\geq 0}\langle v_1v_1^t,\dots,v_nv_n^t,e_ge_g^t \rangle.\] The image $\ifl(\sigma(M))$ is a matroidal cone given by the matrix $[v_1,\dots,v_n,e_g]$. So the image of $\ifl$ is contained in $(\Sigma_{g,\mathrm{co}}^\mathrm{P}[n+1] \cap \Sigma_{g}^\mathrm{P,R}[n+1])/\gl_g(\Z)$, matroidal cones of dimension $n+2$ that have exactly one coloop. Similarly, the deflation map $\dfl$ restricts to $(\Sigma_{g,\mathrm{co}}^\mathrm{P}[n+1] \cap \Sigma_{g}^\mathrm{P,R}[n+1])/\gl_g(\Z)$ and is inverse to $\ifl$.

\subsection{The link of the origin in $A^{\mathrm{trop},\mathrm{P}}_g$}

Given a cone $\sigma \subset \R^g$, we define the link of the origin in $\sigma$ to be $L\sigma = (\sigma-\{0\})/\R_{>0}$, where $\R_{>0}$ acts by scalar multiplication. A face morphism from $\sigma \to \sigma'$ induces a morphism $L\sigma \to L\sigma'$. So if $X = \colim_{i \in \mathcal{I}}\{\sigma_i\}$, then $LX = \colim_{i \in \mathcal{I}}(L(\sigma_i))$. In general, $LX$ is a symmetric CW-complex such that there is a map $B^p \to LX$ for each $p+1$-dimensional cone in the diagram of which $X$ is the colimit.

Let $\Lambda_g = LA^{\mathrm{trop},\mathrm{P}}_g$ be the link of the origin in $A^{\mathrm{trop},\mathrm{P}}_g$ and $\Lambda_g^\mathrm{C}$ the link of the coloop locus $A^\mathrm{C}_g$. Since matroidal cones are simplicial, $\Lambda_g^\mathrm{C}$ is a symmetric $\Delta$-complex. We give a description of it as a functor from $I^\mathrm{op} \to \text{Sets}$. We refer to simplices as cones, with the understanding that we are taking the links of the cones discussed. Choose a representative $M$ for each $\gl_g(\Z)$-orbit of $n+1$-dimensional matroidal cones and choose a representative matrix $A$ for $M$. We have $\Lambda_g^\mathrm{C}([n]) = \{([\sigma_A(M)],\rho):[\sigma_A(M)] \in \Sigma_{g}^\mathrm{P,R}[n]/\gl_g(\Z), \rho: \{\text{extremal rays of }\sigma_A(M)\} \to [n]\}$. Morphisms in $I^\mathrm{op}$ are generated by permutations $\pi:[n] \to [n]$ and inclusions $i_{n-1}:[n-1] \to [n]$, so we describe their images under $\Lambda_g^\mathrm{C}$. The morphism $\Lambda_g^\mathrm{C}(\pi)$ sends $([\sigma_A(M)],\rho)$ to $([\sigma_A(M)],\pi\circ\rho)$. The map $\Lambda_g^\mathrm{C}(i_{n-1})$ sends $([\sigma_A(M)],\rho)$ to $([\sigma_{A'}(M')],\rho')$, where $A'$ is the matrix obtained after removing the last column of $A$, $M'$ is the matroid represented by $A'$ and $\rho'$ is the ordering of extremal rays inherited from $\rho$.

\begin{prop}
Let $\sigma$ be a matroidal cone of rank $< g$ that has no coloop. Then $\sigma$ is a permissible face of $\ifl(\sigma)$.
\end{prop}

\begin{proof}
Suppose $\sigma$ is represented by the $g \times k$ matrix $A = [v_1,\dots,v_k]$. So $\sigma = \R_{\geq 0}\langle v_1v_1^t,\dots,v_kv_k^t \rangle$. Since $\sigma$ has rank $<g$, we may choose $A$ so that the last row is all zero. So $\ifl(\sigma)$ is represented by the matrix $[A|e_g]$. We check the two conditions in Definition \ref{defn: permissible face}.

Suppose $j:[k-1] \to [k]$ is an injection that glues $[\sigma]$ to $[\ifl(\sigma)]$. Suppose $l\in [k]$ is the only element not in the image of $j$. Then $[\sigma] = [\sigma_{A'}]$ where $A'$ is the matrix $A$ with the $l$th column removed. If $l\neq k$, then $[\sigma]$ has a coloop, as represented by the last column of $A'$, which is a contradiction.

By the proof of \cite[Lemma 5.13]{brandt2020top}, there is a natural bijection between $\aut(\sigma)$ and $\aut(\ifl(\sigma))$. By the previous paragraph, there is a map from $\aut(\ifl(\sigma)) \to \aut(\sigma)$, so it is an ismorphism.
\end{proof}

\begin{prop}
The set $S = \{([\sigma],[\ifl(\sigma)]): \sigma \in \Lambda_g^\mathrm{C}, \sigma \text{ has rank $< g$ and no coloop}\}$ is a permissible acyclic matching on the face poset of $\Lambda_g^\mathrm{C}$.
\label{lem: acyclic matching on Lambda_g^C}
\end{prop}

\begin{proof}
It is clear that $S$ is a permissible matching. We show it is acyclic. Suppose there is a cycle \[[\tau_1] \succ [d(\tau_1)] \prec [\tau_2] \succ [d(\tau_2)] \prec \cdots \prec [\tau_n] \succ [d(\tau_n)] \prec [\tau_1].\] Then $[d(\tau_1)]$ and $[d(\tau_n)]$ are two distinct faces of $\tau_1$, but $\tau_1$ has exactly one coloop, so at least one of $[d(\tau_1)]$ and $[d(\tau_n)]$ has a coloop. Without loss of generality we can assume $[d(\tau_n)]$ has a coloop. But that means $([d(\tau_n)],[\tau_n])$ cannot be in $S$.
\end{proof}

\begin{lem}
Let $\sigma(M)$ be a representative of an isomorphism class of cones in $\Lambda_g^\mathrm{C}$ such that $[\sigma(M)]$ is not contained in any matching in $S$. Then either $\sigma(M)$ represents the unique orbit of rank 1 matriodal cones or $\aut(\sigma(M))$ contains a reflection.
\label{lem: Lambda_g^C reflection}
\end{lem}

\begin{proof}
If $M$ contains 0 or 1 coloop, then $\sigma(M)$ is matched in $S$, unless $M$ is equivalent to the $1 \times 1$ matrix $[1]$. Otherwise $\sigma(M)$ has at least 2 coloops. There is an automorphism of $\sigma(M)$ that corresponds to exchanging those two elements, which is a reflection.
\end{proof}

\begin{proof}[Proof (of Theorem \ref{thm: coloop contractible})]
By Theorem \ref{thm: main thm combo}, the subcomplex $\Lambda_g^\mathrm{C}$ is homotopy equivalent to a symmetric CW-complex ${\Lambda_g^\mathrm{C}}'$ with exactly one $p$-cell $(B^p)^\circ/G$ for each symmetric orbit of $p$-simplices $[\sigma]$ that is not matched in $S$, where $G\subset O(p)$ is finite and isomorphic to $\aut(\sigma)$.

The complex ${\Lambda_g^\mathrm{C}}'$ can be built from the point that corresponds to the unique orbit of rank 1 matroidal cones as an iterated mapping cone, for a finite sequence of maps from $S^{p-1}/\aut(\sigma(M))$, where $\sigma(M)$ represents an isomorphism class of matroidal cones that is not matched in $S$. By Lemma \ref{lem: Lambda_g^C reflection} and \cite[Proposition 5.1]{ACP-symCW}, $S^{p-1}/\aut(\sigma(M))$ is contractible. Following the proof of \cite[Theorem 6.1]{ACP-symCW}, the complex ${\Lambda_g^\mathrm{C}}'$ is homotopy equivalent to a point.
\end{proof}

\bibliographystyle{alpha}
\bibliography{reference}

\end{document}